\documentclass{article}
\usepackage{spconf}

\usepackage{epsfig}
\usepackage{subcaption}
\usepackage{calc}
\usepackage{amssymb}
\usepackage{amstext}
\usepackage{amsmath}
\usepackage{amsthm}
\usepackage{pslatex}
\usepackage{tikz}
\usepackage{pgfplots}

\usepackage{amsmath,amsfonts,bm}

\def\1{\bm{1}}

\def\vzero{{\bm{0}}}

\def\vtheta{{\bm{\theta}}}

\def\vbeta{{\bm{\beta}}}
\def\vlambda{{\bm{\lambda}}}
\def\vrho{{\bm{\rho}}}

\def\vu{{\bm{u}}}

\def\vx{{\bm{x}}}

\def\mPhi{{\bm{\Phi}}}

\def\mSigma{{\bm{\Sigma}}}

\def\mPhi{{\bm{\Phi}}}

\DeclareMathAlphabet{\mathsfit}{\encodingdefault}{\sfdefault}{m}{sl}
\SetMathAlphabet{\mathsfit}{bold}{\encodingdefault}{\sfdefault}{bx}{n}

\def\gA{{\mathcal{A}}}
\def\gB{{\mathcal{B}}}
\def\gC{{\mathcal{C}}}

\def\gN{{\mathcal{N}}}

\def\gS{{\mathcal{S}}}
\def\gT{{\mathcal{T}}}

\def\gW{{\mathcal{W}}}

\def\sR{{\mathbb{R}}}

\newcommand{\Var}{\mathrm{Var}}

\newcommand{\Cov}{\mathrm{Cov}}

\usepackage{empheq}

\DeclareMathOperator{\supp}{sup}
\DeclareMathOperator{\convinlaw}{\xrightarrow{\mathcal{L}}}
\theoremstyle{plain}
\newtheorem{theorem}{Theorem}[section]
\newtheorem{proposition}[theorem]{Proposition}

\newtheorem{corollary}[theorem]{Corollary}

\theoremstyle{definition}
\newtheorem{definition}[theorem]{Definition}
\newtheorem{assumption}[theorem]{Assumption}
\theoremstyle{remark}
\usepackage{comment}
\usepackage{flushend}

\def \as {\xrightarrow[]{\text{a.s.}}}

\begin{document}

\title{On the Accuracy of Hotelling-Type Tensor Deflation:\\ A Random Tensor Analysis}

\name{Mohamed El Amine Seddik, Maxime Guillaud, Alexis Decurninge}
\address{Mathematical and Algorithmic Sciences Laboratory, Huawei Technologies France}

\maketitle

\abstract{Leveraging on recent advances in random tensor theory, we consider in this paper a rank-$r$ asymmetric spiked tensor model of the form $\sum_{i=1}^r \beta_i\gA_i + \gW$ where $\beta_i\geq 0$ and the $\gA_i$'s are rank-one tensors such that $\langle \gA_i, \gA_j \rangle\in [0, 1]$ for $i\neq j$, based on which we provide an asymptotic study of Hotelling-type tensor deflation in the large dimensional regime. Specifically, our analysis characterizes the singular values and alignments at each step of the deflation procedure, for asymptotically large tensor dimensions. This can be used to construct consistent estimators of different quantities involved in the underlying problem, such as the signal-to-noise ratios $\beta_i$ or the alignments between the different signal components $\langle \gA_i, \gA_j \rangle$.}


\section{\uppercase{Introduction}}\label{sec:introduction}
The analysis of random tensors has attracted significant attention in the last decade since the introduction of the concept of tensor PCA, which generalizes principal component analysis to high-order arrays. The first model introduced in \cite{montanari2014statistical} is the so-called spiked tensor model of the form $\beta \vx^{\otimes d} + \gW/\sqrt{n}$ where $\vx\in \sR^n$ is some high-dimensional unit vector referred to as a \textit{spike}, $\gW$ is a \textit{symmetric} random tensor of order $d$ having standard Gaussian entries and $\beta\geq 0$ is a parameter controlling the signal-to-noise ratio.

Follow-up results have improved the understanding of the behavior of the spiked model and allowed to identify theoretical and/or algorithmic guarantees in terms of efficient signal recovery. In particular, several works  \cite{perry2016statistical,lesieur2017statistical,jagannath2020statistical,chen2021phase,goulart2021random} have focused on the asymptotic (large dimensional) regime $n\to \infty$.
We briefly summarize their main findings as follows: for $d\geq 3$, it has been shown that there exists a statistical threshold $\beta_{stat}=O(1)$ below which it is information-theoretically impossible to recover or even detect the spike, while above $\beta_{stat}$ recovery is theoretically possible finding a critical point of the square loss. Moreover, the asymptotic alignment $\langle \vx, \vu \rangle$ between $\vx$ and a critical point $\vu$ is given in terms of $\beta$. Besides, since almost all tensor problems, e.g. finding the critical points, are NP-hard \cite{hillar2013most}, many researchers were interested in exhibiting an algorithmic threshold for $\beta$ above which recovery could be possible with a polynomial-time algorithm. The authors in \cite{montanari2014statistical} introduced a method for estimating $\vx$ based on tensor unfolding and showed that spike recovery is possible above the algorithmic threshold $\beta_{algo}=O(n^{ \frac{d-2}{4} })$.

These ideas were further generalized to the \textit{asymmetric} spiked tensor model of the form $\beta \vx_1\otimes \cdots \otimes \vx_d + \gW/\sqrt{\sum_i n_i}$, where $\vx_i\in \sR^{n_i}$ are unit vectors and $\gW$ is a random tensor with standard Gaussian i.i.d.\@ entries. In particular, \cite{arous2021long} provided an analysis of the unfolding method for asymmetric tensors and determined the algorithmic threshold to be $\beta_{algo}=O(n^{ \frac{d-2}{4} })$ when $n_i=n$ for all $i$, while \cite{seddik2021random} showed the existence of a statistical threshold $\beta_s = O(1)$ above which a local solution $\vu_i$ of the MLE aligns with the signal, and further quantified the asymptotic alignments $\langle \vx_i, \vu_i \rangle$.

\paragraph*{Contribution:}
In this paper, we address the extension of these ideas to a more general setting, namely asymmetric \emph{low-rank} spiked tensors of the form $\sum_{i=1}^r \beta_i \vx_{i, 1}\otimes \cdots \otimes \vx_{i, d} + \gW/\sqrt{\sum_i n_i}$ where $\vx_{i,j}\in \sR^{n_j}$ are unit vectors.
Specifically, we consider the study of a simple deflation procedure, first introduced by Hotelling in the context of matrix principal component analysis \cite{Hotelling1933} and still used in modern applications such as the recent AlphaTensor model \cite{fawzi2022discovering}, which consists in iterated rank-one approximations followed by subtraction of the estimated rank-one component.

We focus on the case where the spike components are not orthogonal to each other. While the orthogonal case (i.e. $\langle \vx_{i,k}, \vx_{j,k} \rangle=0$ for all $i\neq j$) trivially boils down to the rank-one model studied in \cite{chen2021phase,seddik2021random}, in the considered setting, i.e. $\langle \vx_{i,k}, \vx_{j,k} \rangle \in (0, 1)$ for $i\neq j$, the behavior of the deflation method is more complex, as depicted in Figure~\ref{fig_orth_vs_corr} in the simplified case $\vx_{i,1}=\dots=\vx_{i,d}=\vx_{i}$ for $i=1,2$.

\begin{figure}[t!]
\centering
\includegraphics[width=.4\textwidth]{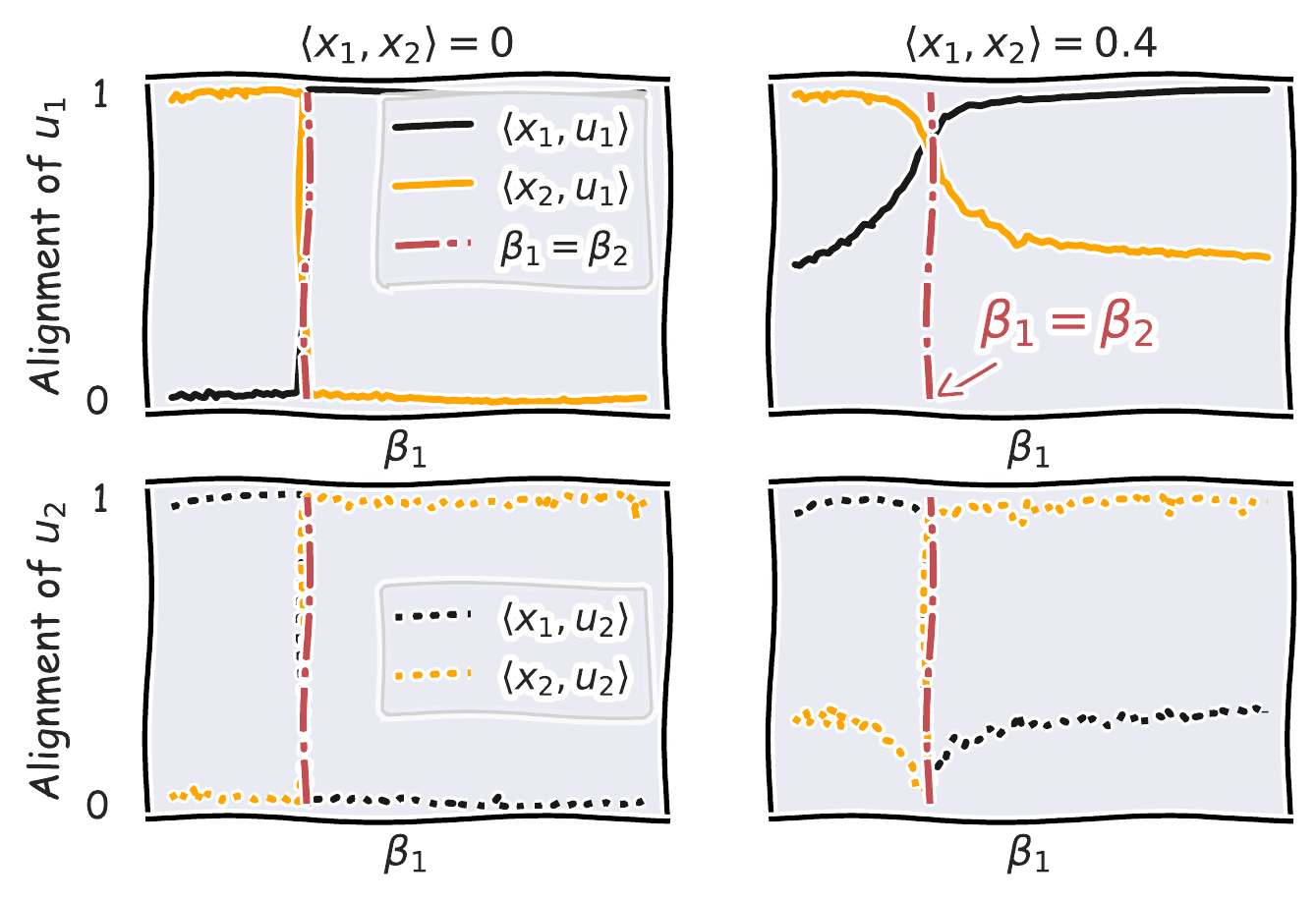}
\caption{Result of tensor deflation applied to a noisy rank-two spiked model $\beta_1\vx_1^{\otimes 3} + \beta_2 \vx_2^{\otimes 3} + \gW/\sqrt{n}$. The first row corresponds to the first deflation step (yielding $\vu_1$) while the second row corresponds to the second deflation step (yielding $\vu_2$). The plots depict the alignments $\langle \vx_i, \vu_j\rangle$ in terms of $\beta_1$ for a fixed $\beta_2$. In the orthogonal case $\langle \vx_1, \vx_2 \rangle = 0$ (left column), deflation is shown to successfully identify the strongest and second strongest components, i.e. $\vu_1$ aligns with $\vx_1$ and $\vu_2$ aligns with $\vx_2$ when $\beta_1>\beta_2$, while $\vu_1$ aligns with $\vx_2$ and $\vu_2$ aligns with $\vx_1$ when $\beta_1<\beta_2$. For the correlated setting, on the other hand (right column, $\langle \vx_1, \vx_2 \rangle = 0.4$), the alignment of $(\vu_1,\vu_2)$ with $(\vx_1,\vx_2)$ is imperfect in the region where $\beta_1 \approx \beta_2$.}
\label{fig_orth_vs_corr}
\end{figure}

In this work, we characterize this behavior by estimating the alignments $\langle \vx_{i,k}, \vu_{j,k} \rangle$ in the asymmetric case, in the high-dimensional regime when $n_i\to \infty$, using the random matrix approach developed in \cite{seddik2021random}. Furthermore, as a by-product of this analysis, we introduce a method to consistently estimate the underlying SNRs $\beta_i$ and the alignments $\langle \vx_{i,k}, \vx_{j,k} \rangle$ and $\langle \vx_{i,k}, \vu_{j,k} \rangle$ from the quantities computed at each step of the deflation procedure.

\paragraph*{Notations:} Scalars are denoted by lowercase letters. Vectors are denote by bold lowercase letters. Tensors are denoted as $\gA, \gB, \gC$. $T_{i_1\ldots i_d}$ denotes the entry $(i_1, \ldots, i_d)$ of tensor $\gT$. The inner product between two order-$d$ tensors $\gA$ and $\gB$ is denoted $\langle \gA, \gB \rangle = \sum_{i_1,\ldots, i_d} A_{i_1\ldots i_d}B_{i_1\ldots i_d}$. The $\ell_2$-norm of $\gA$ is $\Vert \gA\Vert = \sqrt{ \langle \gA, \gA \rangle }$. For any vectors $\vu_{1},\dots, \vu_{d}$, contractions of a tensor $\gA$ are denoted by $\gA(\vu_{1},\dots,\vu_d)=\sum A_{i_1\ldots i_d}u_{1i_1}\dots u_{di_d}$. The notation $\Vert \gA\Vert_{op} = \supp_{\Vert \vu_i \Vert = 1} \vert \gA(\vu_{1},\dots,\vu_d)\vert$ stands for the spectral norm. $[n]$ denotes the set $\{1,\ldots,n\}$.

\section{\uppercase{Model and Main Results}}
We start by describing formally our considered model. Let $r\geq 1$ and $d\geq 3$, we consider the following rank-$r$ order-$d$ spiked tensor model
\begin{align}
\gT_1 = \sum_{i=1}^r \beta_i \vx_{i,1} \otimes \cdots \otimes \vx_{i,d} + \frac{1}{\sqrt n} \mathcal \gW
\end{align}
where $W_{i_1\ldots i_d}\sim \gN(0, 1)$ i.i.d., $\vx_{i, j} \in\sR^{n_j}$ are unit vectors, $n = \sum_{i=1}^d n_i$ and $\beta_i \geq 0$.
\paragraph*{Tensor deflation model:} In order to recover the underlying signal components (i.e. the terms $\beta_i \vx_{i,1} \otimes \cdots \otimes \vx_{i,d}$), Hotelling deflation consists in successive rank-one approximations followed by subtraction of the rank-1 estimate. This can be implemented by computing $\gT_{2},\gT_{3},\dots$ sequentially through
\begin{align}
\gT_{i+1} = \gT_i - \hat\lambda_i \hat\vu_{i, 1} \otimes \cdots \otimes \hat\vu_{i, d} \quad \mathrm{for} \quad i\in [r]
\end{align}
where $\hat\lambda_i \hat\vu_{i, 1} \otimes \cdots \otimes \hat\vu_{i, d}$ is a critical point of the loss $\big\Vert \gT_i -  \lambda_i \vu_{i, 1} \otimes \cdots \otimes \vu_{i, d} \big\Vert_F^2$ which corresponds to the rank-one approximation of $\gT_i$ \cite{lim2005singular}.
The critical points satisfy the Karush-Kuhn-Tucker conditions derived from the Lagrangian of the latter objective, i.e.
\begin{align}\label{eq_kkt}
\gT_i (\hat\vu_{i, 1}, \ldots, \hat\vu_{i, j-1}, \cdot, \hat\vu_{i, j+1} ,\ldots, \hat\vu_{i, d} )= \hat\lambda_i \hat\vu_{i, j}
\end{align}
with $\Vert \vu_{i,j}\Vert = 1$, for all $j\in[d]$.
In the following, we aim at computing the limits of $\hat\lambda_i$ and $\langle\vx_{i,k}, \hat\vu_{j,k}\rangle $ for $i,j\in[r]$ and $k\in[d]$ when the dimensions $n_i$ grow large. For notational convenience, in the sequel the notation $\lim Q$ stands for the limit of the quantity $Q$ when $n_i\to \infty$.

\paragraph*{Sketch of the analytical approach:} We follow the approach developed in \cite{seddik2021random}, whereby each tensor $\gT_i $ is associated to a structured random matrix $\mPhi_d(\gT_i, \hat\vu_{i,1}, \ldots, \hat\vu_{i,d}) \in \sR^{n\times n}$ where the mapping $\mPhi_d$ is defined in \cite[Section 5]{seddik2021random}. Then, the characterization of the limits of $\hat\lambda_i$ and the alignments $\langle\vx_{i,k}, \hat\vu_{j,k}\rangle $ when $n_i\to \infty$ boils down to the computation of the Stieltjes transform of the limiting spectral measure of $\mPhi_d(\gT_i, \hat\vu_{i,1}, \ldots, \hat\vu_{i,d})$. Hence, we need the following definition and technical assumptions.
\begin{definition}\label{def_measure} Let $\mu$ be the probability measure with Stieltjes transform $g(z) = \sum_{i=1}^d g_i(z)$ verifying $\Im[g(z)] > 0$ for $\Im[z] > 0$, where $g_i(z)$ satisfies $g_i^2(z) - (g(z) + z) g_i(z) - c_i = 0$, for $z\notin \gS(\mu)$ and $\gS(\mu)$ stands for the support of $\mu$.
\end{definition}
\begin{assumption}\label{assumptions} We assume that as $n_i\to \infty$, $r=O(1)$ and denote $c_i = \lim \frac{n_i}{\sum_{j=1}^d n_j}$. We further assume that there exists a sequence of critical points such that $\hat\lambda_i\as \lambda_i$, $ \vert\langle \vx_{i,k}, \hat\vu_{j,k} \rangle\vert\as \rho_{ijk}$ and $\vert\langle \hat\vu_{i,k}, \hat\vu_{j,k} \rangle\vert\as\eta_{ijk}$ such that $\lambda_i \notin \gS(\mu)$  and $\rho_{ijk} > 0$.
\end{assumption}
We therefore have the following result\footnote{The proof of Theorem \ref{theorem_spectrum} follows similar arguments as in \cite{seddik2021random} and requires some additional arguments for controlling the statistical dependencies between the $\hat\vu_{i,j}$'s and the noise $\gW$.} which characterizes the limiting spectral measure of $\mPhi_d(\gT_i, \hat\vu_{i,1}, \ldots, \hat\vu_{i,d})$.
\begin{theorem}\label{theorem_spectrum} Under Assumption \ref{assumptions}, the empirical spectral measure of $\mPhi_d(\gT_i, \hat\vu_{i,1}, \ldots, \hat\vu_{i,d})$ converges to the deterministic measure $\mu$ defined in Definition \ref{def_measure}.
\end{theorem}
As shown in \cite{seddik2021random}, in the case $c_i=\frac1d$ for all $i\in [d]$, the measure $\mu$ describes a semi-circle or Wigner-type law of compact support $\gS(\mu) = [-2\sqrt{ \frac{d-1}{d} }, 2\sqrt{ \frac{d-1}{d} }]$, the Stieltjes transform of which writes explicitly as
\begin{align}\label{Stieltjes_transform_g}
g(z) = \frac{-zd + d \sqrt{ z^2 - \frac{ 4(d-1) }{d} } }{ 2 (d-1)}, \quad z\notin \gS(\mu)
\end{align}

\paragraph*{Limiting spectral norms and alignments:} We introduce the quantities $\alpha_{ijk} =  \lim \vert\langle \vx_{i,k}, \vx_{j,k} \rangle\vert, \, f(z) = z + g(z), \, h_i(z) = - \frac{c_i}{g_i(z) }$ that shall be used subsequently. The main result brought by this paper describes the asymptotic singular values and alignments obtained after each tensor deflation step as stated by the following theorem.
\begin{theorem}\label{main_theorem}
Assume that Assumption \ref{assumptions} holds. Then, $\lambda_i$, $\rho_{ijk}$ and $\eta_{ijk}$ satisfy the following system of equations
\begin{small}
\begin{align*}
\begin{dcases}
f( \lambda_j) + \sum_{i=1}^{j-1}  \lambda_i \prod_{k=1}^d \eta_{i j k} - \sum_{i=1}^r \beta_i \prod_{k=1}^d \rho_{ijk} = 0 \text{\ , $1\leq j \leq r$}\\
h_\ell( \lambda_j ) \rho_{kj\ell} + \sum_{i=1}^{j-1} \lambda_i \rho_{ki\ell} \prod_{m\neq\ell}^d \eta_{ijm} - \sum_{i=1}^r \beta_i \alpha_{ik\ell} \prod_{m\neq\ell}^d \rho_{ijm} = 0\\ \text{ $1\leq \ell\leq d, 1\leq j,k \leq r$}\\
h_\ell( \lambda_j ) \eta_{kj\ell} + g_\ell(\lambda_k) \prod_{m\neq \ell}^d \eta_{kjm} + \sum_{i=1}^{j-1} \lambda_i \eta_{ik\ell} \prod_{m\neq \ell}^d \eta_{ijm}  +\ldots \\ - \sum_{i=1}^r \beta_i \rho_{ik\ell} \prod_{m\neq \ell}^d \rho_{ijm} = 0 \text{,\ $1\leq \ell\leq d, 1\leq j<k \leq r$}\\
\end{dcases}
\end{align*}
\end{small}
\end{theorem}
\begin{proof}[Sketch of the proof] We use similar arguments as in \cite{seddik2021random}.
We first show that $\Var[\hat\lambda_j] = O(n^{-1})$ and use a concentration argument to show that $\hat\lambda_j$ concentrates around its expectation. Simlarly, the same property holds for the alignments.
Then, we evaluate the expectation of the scalar product between \eqref{eq_kkt} and $\vx_i$ or $\hat\vu_i$ using Stein's Lemma\footnote{$\mathbb{E}[Wf(W)] = \mathbb{E}[f'(W)]$ for $W\sim\mathcal{N}(0, 1)$.}.
\end{proof}

\begin{figure}[t!]
\includegraphics[width=.49\textwidth]{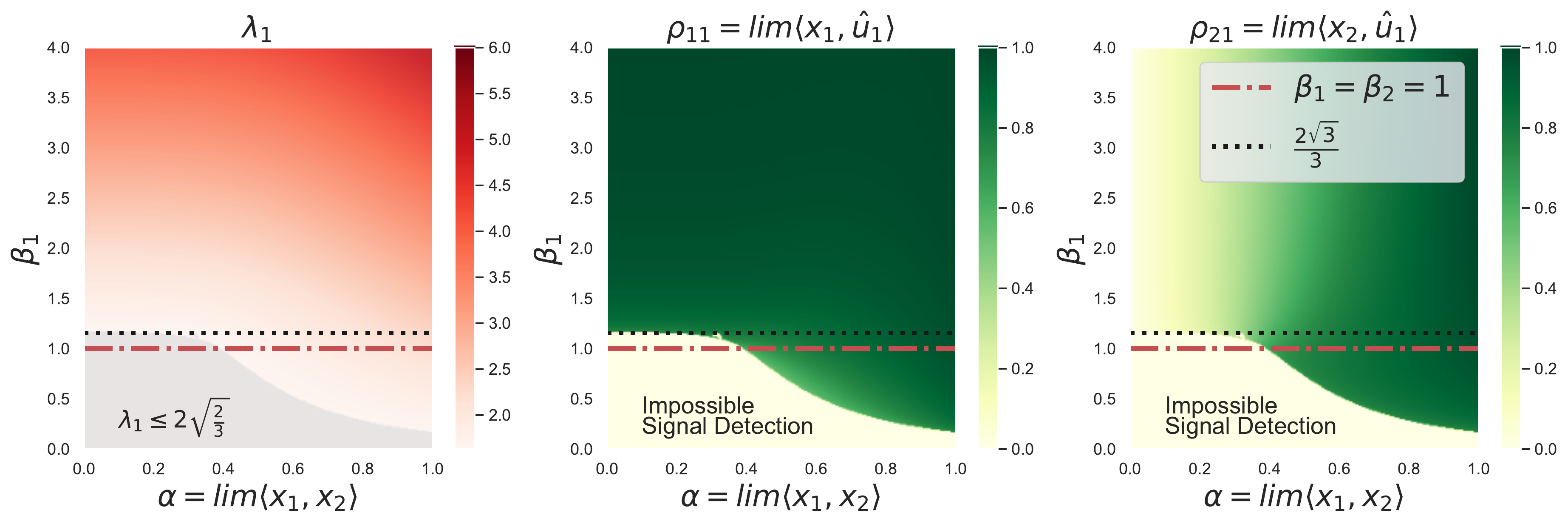}
\includegraphics[width=.49\textwidth]{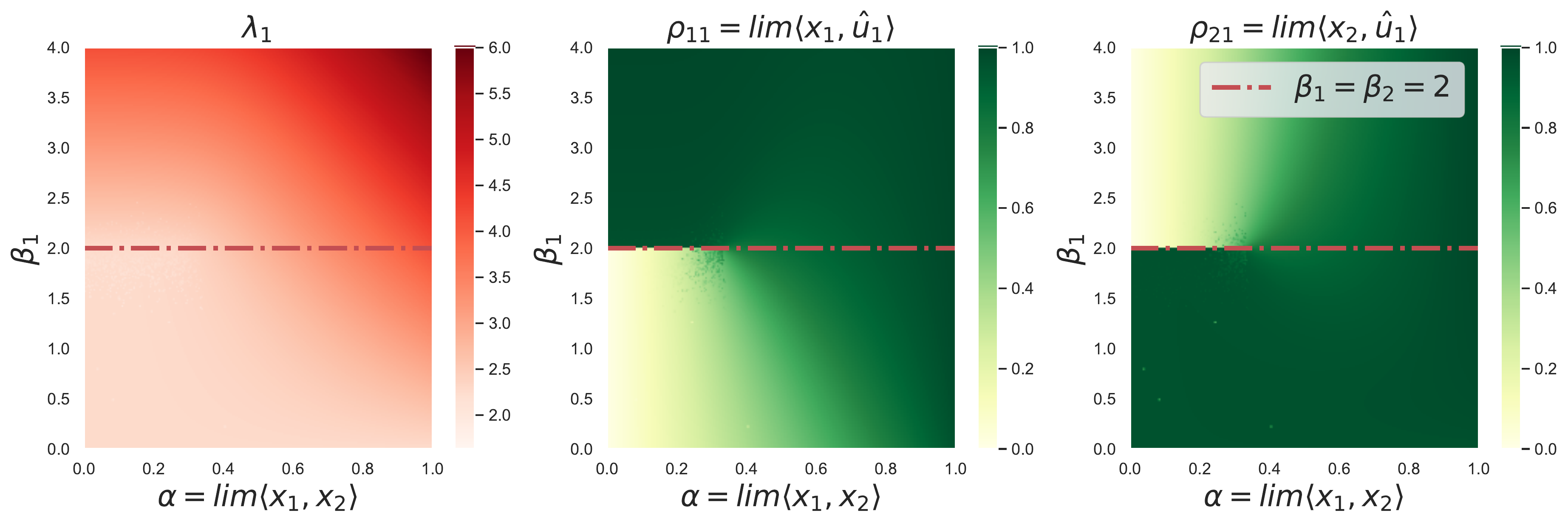}
\caption{Phase diagram of the two spikes model in \eqref{eq_two_spikes_model}. First row corresponds to $\beta_2 = 1$ and second row for $\beta_2=2$. The first column depicts $\lambda_1$ varying $\beta_1$ and $\alpha$, while the remaining columns depict the asymptotic alignments between the $\vx_i$'s and $\hat\vu_1$. The figures were obtained by solving the first three equations in \eqref{eq_system_two_spikes}.}
\label{fig_phase}
\end{figure}

\begin{figure*}[t!]
\begin{tikzpicture}
\node[anchor=south west] (img) {\includegraphics[width=\textwidth]{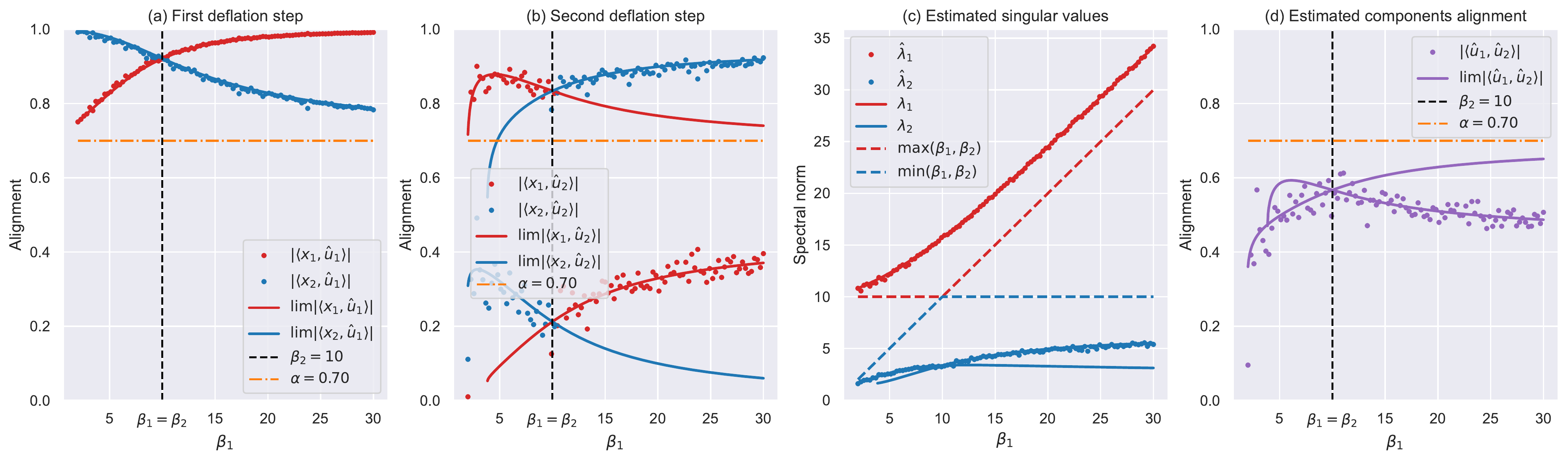}};
\end{tikzpicture}
\caption{Illustration of $\langle \vx_i, \hat\vu_j \rangle$ for $(i,j)\in\{1, 2 \}$, $\langle \hat\vu_1, \hat\vu_2 \rangle$, $\hat\lambda_1$ and $\hat\lambda_2$ vs. their limits for $\beta_2=10$ and $\alpha = \langle \vx_1, \vx_2 \rangle = 0.7$ of the two spikes tensor model in \eqref{eq_two_spikes_model}. (a) shows the alignments between the signal components and the first singular vectors corresponding to the best rank-one approximation of $\gT_1$. (b) shows the alignments with the second singular vectors computed after deflation. (c) depicts the singular values. (d) shows the alignments between the singular vectors computed at each step of the deflation procedure. Simulations were performed on a tensor of dimensions $(50, 50, 50)$.}
\label{figure_alignments}
\end{figure*}

\begin{figure}[b!]
\includegraphics[width=.49\textwidth]{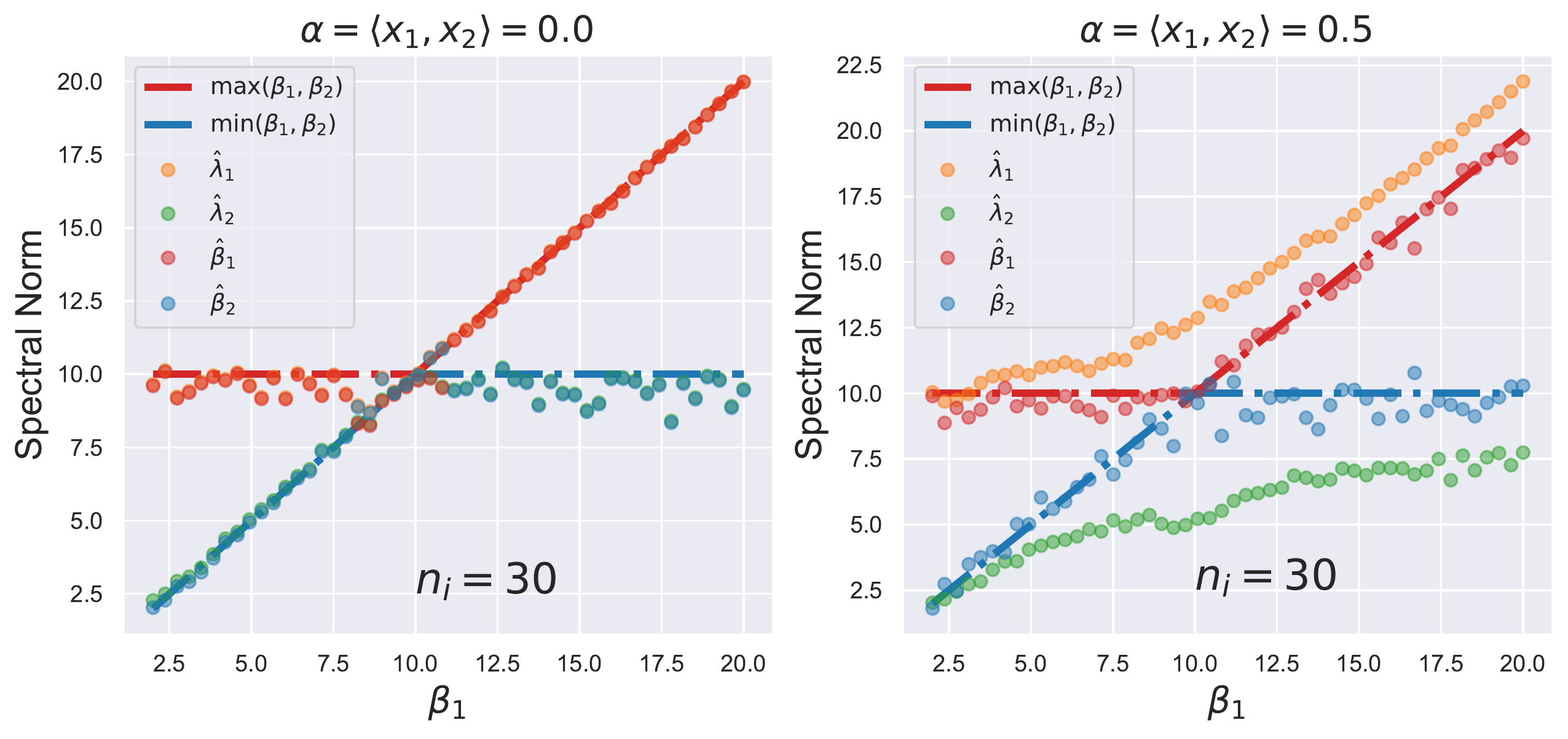}
\includegraphics[width=.49\textwidth]{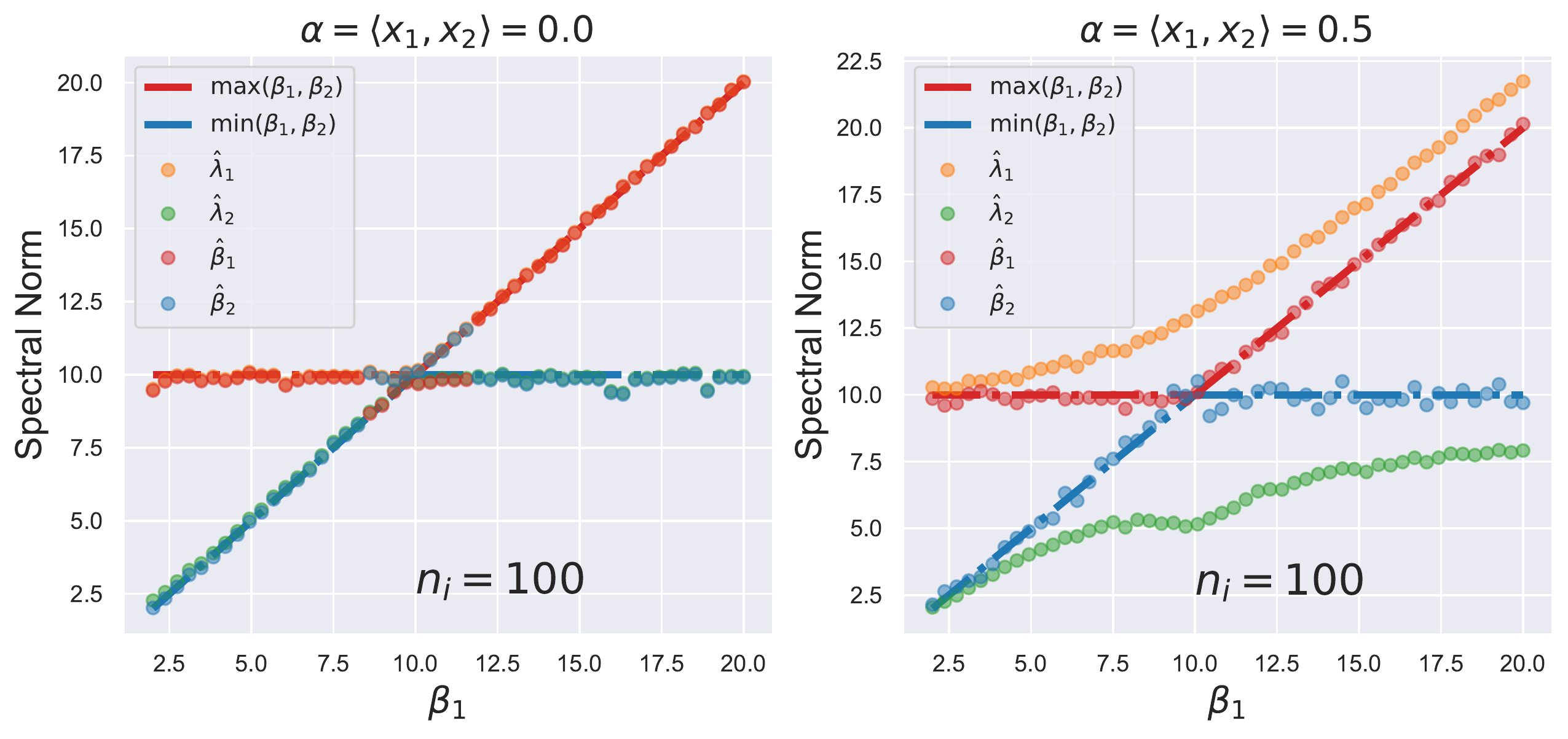}
\caption{Estimation of the $\beta_i$'s by solving the system $\psi(\hat\vlambda,\cdot,\cdot)=\boldsymbol{0}$ given $\hat\lambda_1$, $\hat\lambda_2$ and $\hat \eta = \langle \hat\vu_1, \hat\vu_2 \rangle$ estimated from a two steps deflation on a tensor distributed as in \eqref{eq_two_spikes_model}. First row corresponds to $n_i=30$ and second row to $n_i=100$.}
\label{fig_estimation_beta}
\end{figure}

\paragraph*{Particular case of a rank-2, order-3 tensor:} For the sake of clarity, let us consider the example of a rank-$2$ order-$3$ spiked tensor with $n_1=n_2=n_3$, thus
\begin{align}\label{eq_two_spikes_model}
\gT_1 = \sum_{i=1}^2 \beta_i \vx_{i, 1} \otimes \vx_{i, 2} \otimes \vx_{i, 3} + \frac{1}{\sqrt{n}}\gW
\end{align}
Furthermore, we assume that for all $i\neq j$ and each $k\in[3]$, $\lim\vert \langle \vx_{i,k}, \vx_{j,k} \rangle \vert = \alpha\in [0, 1]$. In this case, since all the dimensions $n_i$ are equal, the limits of $\vert \langle \vx_{i,k}, \hat\vu_{j,k} \rangle \vert$ and of $\vert \langle \hat\vu_{1,k}, \hat\vu_{2,k} \rangle \vert$ are both independent from $k$ by symmetry. Therefore, we drop their dependence on $k$ in our notations. Hence, the system of equations in Theorem \ref{main_theorem} reduces to seven equations, detailed in the following corollary.
\begin{corollary}\label{corollary} Denote $\rho_{ij} = \lim \vert \langle \vx_{i,k}, \hat\vu_{j,k} \rangle \vert$ for $i,j\in [2]$ and $\eta = \lim \vert \langle \hat\vu_{1,k}, \hat\vu_{2,k} \rangle \vert$ and suppose that Assumption \ref{assumptions} holds, then $ \lambda_i$, $\rho_{ij}$ and $\eta$ satisfy $\psi(\vlambda,\vbeta,\vrho)=\boldsymbol{0}$ with $\vlambda=(\lambda_1,\lambda_2,\eta)$, $\vrho=(\rho_{11},\rho_{12},\rho_{21},\rho_{22})$, $\vbeta=(\beta_1,\beta_2,\alpha)$ and
\begin{small}
\begin{equation}\label{eq_system_two_spikes}
\psi(\vlambda,\vbeta,\vrho) =
\left(\begin{array}{c}
f(\lambda_1) - \beta_1 \rho_{11}^3 - \beta_2 \rho_{21}^3 \\
h( \lambda_1) \rho_{11} - \beta_1 \rho_{11}^2 - \beta_2 \alpha \rho_{21}^2\\
h( \lambda_1) \rho_{21} - \beta_1 \alpha \rho_{11}^2 - \beta_2 \rho_{21}^2\\
f( \lambda_2 ) +  \lambda_1 \eta^3 - \beta_1 \rho_{12}^3 - \beta_2 \rho_{22}^3\\
h( \lambda_2 ) \rho_{12} +  \lambda_1 \rho_{11} \eta^2 - \beta_1 \rho_{12}^2 - \beta_2 \alpha \rho_{22}^2 \\
h( \lambda_2 ) \rho_{22} +  \lambda_1 \rho_{21} \eta^2 - \beta_1 \alpha \rho_{12}^2 - \beta_2 \rho_{22}^2 \\
h( \lambda_2) \eta + q( \lambda_1 ) \eta^2 - \beta_1 \rho_{11} \rho_{12}^2 - \beta_2 \rho_{21} \rho_{22}^2
\end{array}\right)
\end{equation}
\end{small}
where $h(z) = \frac{-1}{g(z)}$ and $q(z) = z + \frac{g(z)}{3}$ with $g(z)$ given by \eqref{Stieltjes_transform_g} for $d=3$ and we recall that $f(z) = z + g(z)$.
\end{corollary}

Fixing $\vbeta=(\beta_1,\beta_2,\alpha)$, one can solve $\psi(\vlambda,\vbeta,\vrho)=\boldsymbol{0}$ in $(\vlambda,\vrho)$ while ensuring that $0\leq \eta,\rho_{ij} \leq 1$ and $\lambda_1, \lambda_2> 2\sqrt{\frac{2}{3}}$.
This provides a fixed point equation satisfied by the asymptotic limits of the spectral norms $\hat\lambda_i$ and the alignments $ \langle \vx_{i,k}, \hat\vu_{j,k} \rangle $ and $\langle \hat\vu_{1,k}, \hat\vu_{2,k} \rangle$.

Note that the first three equations in \eqref{eq_system_two_spikes} only involve $ \lambda_1$, $\rho_{11}$ and $\rho_{21}$ and are decoupled from the last four equations. Therefore, solving them allows to obtain the \textit{phase diagram} related to the dominant singular mode $(\hat\lambda_1,\hat\vu_1)$, depicted in Figure~\ref{fig_phase}. It shows that when $\beta_2$ is not large enough (e.g. $\beta_2=1$, top row), there exists a region (varying $\beta_1$ and $\alpha$) where it is information-theoretically impossible to detect a signal, while outside this region estimation becomes possible with the MLE (in that case, the estimated $\hat\vu_1$ is shown to be correlated with both $\vx_1$ and $\vx_2$).
For $\beta_2$ sufficiently large (e.g. $\beta_2=2$, bottom row of Figure~\ref{fig_phase}), signal detection is always possible and the singular vector $\hat\vu_1$ presents a higher alignment with the signal components having the highest SNR $\beta_i$ (see the two columns on the right).

Moreover, we illustrate  in Figure~\ref{figure_alignments} the matching between $\hat\vlambda=(\hat\lambda_1, \hat\lambda_2, \hat\eta )$, $\hat\vrho=(\hat\rho_{11},\hat\rho_{12},\hat\rho_{21},\hat\rho_{22})$ (where the rank-one approximations are performed using tensor power iteration initialized by tensor SVD \cite{auddy2022estimating}) and their asymptotic limits $\vlambda,\vrho$.
Note that, for some $\beta_1$, the equation $\psi(\cdot,\vbeta,\cdot)=\boldsymbol{0}$ has two distinct solutions. These two solutions correspond to two different sequences of critical points.
However, in practice, the chosen initialization favor one sequence of critical points as can be seen in Figure~\ref{figure_alignments}.

\section{Consistent SNR estimation}
Having set the relationship between the $\beta_i$'s and the limits of the different spectral norms and alignments in our problem, one can exploit this mapping to design a consistent estimator of the underlying SNRs $\beta_i$'s. Indeed, $\hat\vlambda$ can directly be estimated from $\hat\vu_i$ obtained with the deflation.

Then we denote $\hat\vbeta$ and $\hat\vrho$ the estimates of $\vbeta$ and $\vrho$ as the vectors satisfying $\psi(\hat\vlambda,\hat\vbeta,\hat\vrho)= \boldsymbol{0}$ with $\psi$ defined in \eqref{eq_system_two_spikes}. The additional condtions required for the existence and the uniqueness of such solution is not studied in the present paper and shall be considered in an extended version. In our simulations, we find one solution for $\psi(\hat\vlambda,\cdot,\cdot)=\boldsymbol{0}$ if $\hat\lambda_1, \hat\lambda_2>2 \sqrt{2/3}$.
We illustrate the result of such estimation in Figure~\ref{fig_estimation_beta} where we see that $(\hat\beta_1,\hat\beta_2)$ consistently estimate $(\max(\beta_1,\beta_2),\min(\beta_1,\beta_2))$,  while the naive estimator $(\hat\beta_1,\hat\beta_2)=(\hat\lambda_1,\hat\lambda_2)$ exhibits a large error in the non-orthogonal case ($\alpha=0.5$, second column). Moreover, as the dimensions of the tensor increase, Fig.~\ref{fig_estimation_beta} shows (comparing first and second row) that the estimation is consistent. This consistency can be related to a classical concentration phenomenon  \cite{benaych2011fluctuations}.

\section{Conclusion}
We have provided an analysis of a tensor deflation method in the high-dimensional regime and assuming a low-rank spiked tensor model with correlated signal components. Our analysis allows precise description of the asymptotic behavior of such models and provides consistent estimation of its parameters as shown in the last part of the paper. This paves a new way for analysis of more sophisticated tensor decomposition methods and the understanding of more general tensor models through random tensor theory.

\bibliographystyle{IEEEbib}
{\small
\bibliography{example}}

\end{document}